\documentclass[12pt, amsfonts]{amsart}

\usepackage{amsmath,amssymb,amsthm}
\usepackage[colorlinks=true]{hyperref}
\usepackage[all]{xy}
\usepackage{mathtools}

\numberwithin{equation}{section}

\newtheorem{theorem}{Theorem}[section]
\newtheorem{proposition}[theorem]{Proposition}
\newtheorem{lemma}[theorem]{Lemma}
\newtheorem{corollary}[theorem]{Corollary}

\theoremstyle{definition}

\newtheorem{example}[theorem]{Example}

\theoremstyle{remark}

\newcommand{\Z}{\mathbb{Z}}

\newcommand{\C}{\mathbb{C}}
\newcommand{\Hq}{\mathbb{H}}

\newcommand{\Pe}{\mathcal{P}}
\newcommand{\Sq}{\operatorname{Sq}}

\newcommand{\hocolim}{\operatorname{hocolim}}

\newcommand{\CP}{\mathbb{C}\operatorname{P}}

\newcommand{\Top}{\operatorname{Top}}

\title{Steenrod problem and some graded Stanley-Reisner rings}
\author{Masahiro Takeda}
\address{Department of Mathematics, Kyoto University, Kyoto, 606-8502, Japan}
\email{takeda.masahiro.8h@kyoto-u.ac.jp}


\begin{document}

 \maketitle

  \begin{abstract}
    “What kind of ring can be represented as the singular cohomology ring of a space?” is a classic problem in algebraic topology, posed by Steenrod. 
    In this paper, we consider this problem when rings are the graded Stanley-Reisner rings, in other words the polynomial rings divided by an ideal generated by square-free monomials.
    We give a necessary and sufficiently condition that a graded Stanley-Reisner ring is realizable when there is no pair of generators $x,y$ such that $|x|=|y|=2^n$ and $xy\ne 0$.
  \end{abstract}


\section{Introduction}
It is a classical problem in algebraic topology asked by Steenrod in \cite{S} which graded ring occurs as the cohomology ring of a space.
Especially when the graded ring is polynomial ring, this problem was studied by many researchers, for example \cite{AW,A,AG2,CE,DMW,DW1,DW2,H,ST,T}. 
And this polynomial ring case was finally solved by Andersen and Grodal \cite{AG1}.

On the other hand, when the graded ring is a monomial ideal ring, in other words a polynomial ring divided by an ideal generated by monomials, some researchers studied this problem.
The realizability of Stanley-Reisner rings, square-free monomial ideal rings, generated by degree $2$ elements is proved by Davis and Januszkiewicz in \cite{DJ}.
Trevisan generalize their construction and prove the realizability of monomial ideal rings generated by degree $2$ elements in \cite{Tr}.
By using polyhedral products, the realizability of Stanley-Reisner rings of a certain class is proved by Bahri, Bendersky, Cohen and Gitler in \cite{BBCG}.
So and Stanley prove the realizability of graded monomial ideal ring modulo torsion in \cite{SS}.
Thus there are results about the realizability of monomial ideal rings, but there are few results about necessary conditions for monomial ideal rings to be realizable.

In this paper we obtain a necessary and sufficiently condition for a graded Stanley-Reisner ring to be realizable when there is no pair of generators $x,y$ such that $|x|=|y|=2^n$ and $xy\ne 0$.
At first, we define the graded Stanley-Reisner ring. 
Let $K$ be a simplicial complex with the vertex set $V$, and $\phi \colon V\rightarrow 2\Z_{>0}$.
Then the graded Stanley-Reisner ring $SR(K,\phi)$ is defined by
\[
    SR(K,\phi)\cong \Z[V]/I,
\] 
where $\Z[V]$ is the polynomial ring generated by $x \in V$ with $|x|=\phi(x)$ and $I$ is the ideal generated by monomials $x_1 x_2\dots x_k$ with $\{ x_1,x_2,\dots ,x_k\}\notin K$ as a simplex.

To state the main theorem in this paper we set notation.
A simplex of a simplicial complex is maximal when the simplex is not a face of a larger simplex in the simplicial complex.
For a simplicial complex $K$ with the vertex set $V$, we define a poset (not subcomplex) $P_{\max}(K) \subset K$, where we regard $K$ as a subset of the power set of $V$.
For $\sigma \in K$, $\sigma \in P_{\max}(K)$ if and only if there exist maximal simplices $\sigma_1, \dots \sigma_n \in K$ such that $\bigcap \sigma_i =\sigma$.
And for $\sigma,\tau \in P_{\max}$, $\sigma<\tau$ when $\sigma$ is a face of $\tau$ in $K$.

\begin{theorem}\label{main}
    Let $SR(K,\phi)$ be the graded Stanley-Reisner ring for a simplicial complex $K$ with the vertex set $V$ and $\phi \colon V\rightarrow 2\Z_{>0}$.
    Suppose that the graded Stanley-Reisner ring $SR(K,\phi)$ satisfies the following: 
    \begin{itemize}
        \item If generators $x,y\in V$ satisfy $\phi(x)=\phi(y)=2^i$ for some $i\geq 2$, then $xy=0$ in $SR(K,\phi)$.
    \end{itemize}
    Then there is a space $X$ such that $H^*(X)\cong SR(K,\phi)$ if and only if $SR(K,\phi)$ satisfies the following condition.
    \begin{itemize}
        \item For $\sigma\in P_{\max}(K)$ the set $\{\phi(x)\mid x\in \sigma \}$ is equal to $\{2,2,\dots , 2\}$, $\{4,6,\dots 2n+2\}\cup \{2,2,\dots , 2\}$ or $\{4,8,\dots 4n\}\cup \{2,2,\dots , 2\}$ as a multiset for some $n$.
    \end{itemize}
\end{theorem}
This is the main theorem in this paper.
We believe that the unnatural assumption in the main theorem can be replaced by the following condition.
\begin{itemize}
    \item If generators $x,y\in V$ satisfies $\phi(x)=\phi(y)=4$, then $xy=0$ in $SR(K,\phi)$.
\end{itemize}
But now we have yet to prove the theorem under this condition.
The reason why the assumption required is in the latter part of this paper.

We can generalize the construction of a space $X$ with $H^*(X)$ being isomorphic to the graded Stanley-Reisner ring to a wider classes. The following theorem is proved in Section \ref{s_hocolim}.
\begin{theorem}\label{hocolim1}
    Let $SR(K,\phi)$ be the graded Stanley-Reisner ring  for simplicial complex $K$ with the vertex set $V$ and $\phi \colon V\rightarrow 2\Z_{>0}$.
    If $SR(K,\phi)$ satisfies the following condition, we can construct a space $X$ as a homotopy colimit such that $H^*(X)\cong SR(K,\phi)$. 
    \begin{itemize}
        \item There is a decomposition of $V$, $\coprod_{i}A_{i} = V$, such that for all $i$ and $\sigma\in P_{\max}(K)$ the set $\{\phi(x)\mid x\in \sigma \cap A_i \}$ is equal to $\{2,2,\dots , 2\}$, $\{4,6,\dots 2n\}\cup \{2,2,\dots , 2\}$ or $\{4,8,\dots 4n\}\cup \{2,2,\dots , 2\}$ as a multiset for some $n$.
    \end{itemize}
\end{theorem}

In the first half of this paper, Sections \ref{s_preliminary}, \ref{s_hocolim}, \ref{s_example}, we construct a space $X$ with $H^*(X)$ isomorphic to a graded Stanley-Reisner ring and prove Theorem \ref{hocolim1}.
In the latter half, Sections \ref{s_Steenrod}, \ref{s_requirement}, we obtain the necessary condition that graded Stanley-Reisner rings occur as the cohomology ring of a space.
At last, by combining these results we prove the main theorem in Section \ref{s_proof}.

\section*{Acknowledgement}

The author is grateful to Donald Stanley for suggesting this issue and for valuable advice. 
The author is supported by JSPS KAKENHI Grant Numbers 21J10117.


\section{Homotopy colimit}\label{s_preliminary}

In this section we recall a homotopy colimit and prove some lemmas we will use.

Let $P$ be a poset.
The order complex of $P$, $\Delta (P)$, is a simplicial complex whose faces are totally ordered subsets in $P$.
We regard $P$ as a category.
For a functor $F\colon P\rightarrow \mathrm{Top}$, the homotopy colimit is defined as the following
\[
    \hocolim_{P} F = \coprod_{\sigma =(x_1<x_2<\dots < x_k) \in \Delta(P)} |\sigma| \times F(x_k)/ \sim,
\]
where the equivalence is $(\iota (x),y)\sim (x,F(\iota)(y))$ for $\iota \colon \tau \hookrightarrow \sigma$ and $x\in |\tau|, y \in F(\max(\sigma) )$.

We denote $P_{<a}=\{p\in P\mid p<a\}$ and $P_{\leq a}=\{p\in P\mid p\leq a\}$ for some $a\in P$.
\begin{lemma}\label{pushout}
    Let $(P,<)$ be a finite poset and $F\colon P\to \mathrm{Top}$ be a functor.
    Let $a\in P$ be a maximal element.
    Then there is a pushout diagram 
    \[
        \xymatrix{
            \hocolim_{P_{<a}} F \ar[r] \ar[d] & \hocolim_{ P\backslash \{a\}} F \ar[d]\\
            \hocolim_{P_{\leq a}} F \ar[r]& \hocolim_{P} F,
        }
    \]
    where for a subset $P' \subset P$ $\hocolim_{P'} F$ means the homotopy colimit of the functor $F|_{P'}\colon P' \to Top$.
\end{lemma}
\begin{proof}
    By the definition of homotopy colimit, we obtain that 
    \begin{align*}
        \hocolim_{ P\backslash \{a\}} F\bigcup \hocolim_{P_{\leq a}} F &= \hocolim_{P} F\\
        \hocolim_{ P\backslash \{a\}} F\bigcap \hocolim_{P_{\leq a}} F &= \hocolim_{P_{<a}} F
    \end{align*}
    And the inclusions 
    \begin{align*}
        \hocolim_{P_{<a}} F &\hookrightarrow \hocolim_{ P\backslash \{a\}} F\\
        \hocolim_{P_{<a}} F &\hookrightarrow \hocolim_{ P\leq \{a\}} F
    \end{align*}
    are cofibrations. By combining these we obtain this lemma.
\end{proof}

Next, we see the relation between the homotopy pushout and the graded Stanley-Reisner ring.
For a subcomplex $K' \subset K$, let $V(K')$ be the vertex set of $K'$.
\begin{lemma}\label{kernel}
    Let $K$ be a simplicial complex with the vertex set $V$, and $\phi \colon V\rightarrow 2\Z_{>0}$.
    Let $K_1,K_2$ be subcomplexes of $K$. 
    We assume the following:
    \begin{itemize}
        \item There is a space $X$ with $H^*(X)\cong SR(K_1\cap K_2, \phi)$.
        \item There are spaces $X_i$ with $H^*(X_i)\cong SR(K_i,\phi)$ for $i=1,2$.
        \item There are maps $\pi_i \colon X\rightarrow X_i$ such that $\pi_i^*$ is identified with the natural projection $SR(K_i,\phi) \rightarrow SR(K_1\cap K_2,\phi)$ for $i=1,2$ in cohomology.
    \end{itemize}
    Then the cohomology ring of the homotopy pushout of the following diagram 
    \[
        \xymatrix{ X \ar^{\pi_1}[r]\ar^{\pi_2}[d] & X_1 \\
            X_2}
    \]
    is isomorphic to $SR(K_1\cup K_2,\phi)$.
\end{lemma}
\begin{proof}
    Since the maps $\pi_1^*,\pi_2^*$ are surjective in cohomology, 
    by Mayer-Vietoris sequence, the cohomology of the homotopy pushout is the kernel of the map 
    \[
        SR(K_1,\phi)\oplus SR(K_2,\phi)\rightarrow SR(K_1\cap K_2,\phi) \quad x\oplus y \mapsto \pi_1^*(x)-\pi_2^*(y),
    \]
    for $x \in SR(K_1,\phi)$ and $y\in SR(K_2,\phi)$.
    The kernel of this map is the direct sum of the kernel of $\pi_1^*$, the kernel of $\pi_2^*$ and $\{x\oplus x \in SR(K_1,\phi) \oplus SR(K_2,\phi)\mid x\in SR(K_1\cap K_2,\phi)\}$, where $SR(K_1\cap K_2,\phi)$ is regarded as a submodule of $SR(K_i,\phi)$ naturally.
    So we can take the disjoint union of the set $V(K_1)\setminus V(K_1)\cap V(K_2)$, $V(K_2)\setminus V(K_1)\cap V(K_2)$ and 
    $$\{x\oplus x \mid x\in V(K_1)\cap V(K_2)\}$$
    as a ring generator of the kernel, and the ring is isomorphic to $SR(K_1\cup K_2,\phi)$.
\end{proof}



\section{Construction of homotopy colimit}\label{s_hocolim}

In this section we construct a homotopy colimit representation of a space $X$ with $H^*(X)\cong SR(K,\phi)$ for some graded Stanley-Reisner rings $SR(K,\phi)$.
This construction is an analogy to the construction in \cite{DJ}.
The Davis-Januskiewicz space that first appeared in \cite{DJ} is constructed by the union of the products of complex projective spaces.
As far as looking at cohomology, our construction is like a graded version of their construction.

\subsection{Maps between the classifying spaces of Lie groups}
We define maps between the classifying spaces of Lie groups and see some properties.
Let $$\iota_1 \colon SU(n)\rightarrow SU(n+1)$$ be the inclusion $\iota_1(A)=A\oplus 1$ for $A\in SU(n)$, and $$\iota_2 \colon Sp(n)\rightarrow Sp(n+1)$$ be the inclusion $\iota_2(A)=A\oplus 1$ for $A\in Sp(n)$.
For the quaternion $\Hq$ and the set of complex $2\times2$-matrixes $M(2,\C)$, let $c \colon \Hq \rightarrow M(2,\C)$ be the map $c(z+jw)=\begin{pmatrix} z & -\bar{w} \\ w & \bar{z} \end{pmatrix}$ for $z,w\in \C$.
Let $$\iota_3 \colon Sp(n)\rightarrow SU(2n)$$ be the map such that for $A=(a_{i,j})_{ij}\in Sp(n)$, 
\[
    \iota_3(A)=\begin{pmatrix}
        c(a_{1,1}) & c(a_{1,2}) & \dots \\
        c(a_{2,1}) & c(a_{2,2}) & \\
        \vdots & & \ddots
    \end{pmatrix}\in SU(2n).
\]
Since $\iota_i$ is a homomorphism, $\iota_i$ induces the map between classifying map.
We denote these maps as same symbol $\iota_i$.
Since the following diagram is commutative
\[
    \xymatrix{
        Sp(n) \ar^{\iota_2}[r]\ar^{\iota_3}[d] & Sp(n+1)\ar^{\iota_3}[d]\\
        SU(2n) \ar^{\iota_1}[r]& SU(2n+2),
    }
\]
there is a commutative diagram
\[
    \xymatrix{
        BSp(n) \ar^{\iota_2}[r]\ar^{\iota_3}[d] & BSp(n+1)\ar^{\iota_3}[d]\\
        BSU(2n) \ar^{\iota_1}[r]& BSU(2n+2).
    }
\]

We recall the cohomology of these classifying spaces. There is an isomorphism 
\[
    H^*(BSU(n))\cong \Z[c_2,c_3,\dots c_n],
\]
where $c_i$ is the $i$-th Chern class.
By degree reason, we obtain $\iota_{3}^*(c_{2n+1})=0$, and the next lemma holds.
\begin{lemma}[cf. Chapter III, Theorem 5.8 in \cite{TM}]
  There is an isomorphism 
  \[H^*(BSp(n);\Z)\cong \Z[\iota_3^*(c_2),\iota_3^*(c_4),\dots \iota_3^*(c_{2n})].\]
\end{lemma}
In this paper we take the generators of $H^*(BSp(n))$ as in this lemma.
Then there are equations for $\iota_1\colon BSU(n) \rightarrow BSU(n+1)$ and $\iota_2\colon BSp(n) \rightarrow BSp(n+1)$(cf. Chapter III in \cite{TM})
\[
      \iota_1^*(c_i)=\begin{cases}
          c_i \quad &i\leq n\\
          0 \quad &i=n+1,
      \end{cases} \quad 
      \iota_2^*(\iota_3^*(c_{2i}))=\begin{cases}
        \iota_3^*(c_{2i}) \quad &i\leq n\\
          0 \quad &i=n+1.
    \end{cases}  
\]
In summary, $\iota_1,\iota_2$ and $ \iota_3$ are the maps that send each generator to its corresponding generator or $0$ in cohomology. 

\subsection{Construction}
We define a functor by using the maps $\iota_1,\iota_2$ and $ \iota_3$.
Let $K$ be a simplicial complex with the vertex set $V$, and $\phi \colon V\rightarrow 2\Z_{>0}$ satisfying the following condition.
\begin{itemize}
    \item There is a decomposition of $V$, $\coprod_{i}A_{i} = V$, such that for all $i$ and $\sigma\in P_{\max}(K)$, the set $\{\phi(x)\mid x\in \sigma \cap A_i \}$ is equal to $\{2,2,\dots , 2\}$, $\{4,6,\dots 2n+2\}\cup\{2,\dots ,2\}$ or $\{4,8,\dots 4n\}\cup\{2,\dots ,2\}$ as a multiset for some $n$.
\end{itemize}
The simplicial complex $K$ can be regard as a poset by inclusions.
Then we define a subposet $P\subset K$ satisfying the following:
\begin{itemize}
    \item $P_{\max}(K)\subset P$.
    \item For any $\sigma \in P$ and $i$, the set $\{\phi(x)\mid x\in \sigma \cap A_i \}$ is equal to $\{2,2,\dots , 2\}$, $\{4,6,\dots 2n+2\}\cup\{2,\dots ,2\}$ or $\{4,8,\dots 4n\}\cup\{2,\dots ,2\}$ as a multiset for some $n$.
\end{itemize}
Then we regard the poset $P$ as a category and we define a functor $F\colon P \rightarrow \Top$.
For $\sigma \in K$, $$X_\sigma=BSp(n)\times \prod_{\{x\in \sigma \mid \phi(x)=2\}}\CP^{\infty}$$ when $\{\phi(x)\mid x\in \sigma \}=\{4,8,\dots ,4n\}\cup\{2, \dots 2\}$, $$X_\sigma=BSU(n+1)\times \prod_{\{x\in \sigma \mid \phi(x)=2\}}\CP^{\infty}$$ when $\{\phi(x)\mid x\in \sigma \}=\{4,6,\dots ,2n+2\}\cup\{2, \dots 2\}$, and $X_{\sigma}$ is a point when $\sigma$ is the empty set. 
For $\sigma \subset \tau \in K$ let 
\[\iota\colon \prod_{\{x\in \sigma \mid \phi(x)=2\}}\CP^{\infty}\rightarrow \prod_{\{x\in \tau \mid \phi(x)=2\}}\CP^{\infty}\]
be the inclusion such that each vertex corresponds to the same vertex.
Then let $f_{\sigma , \tau}\colon X_{\sigma} \rightarrow X_{\tau}$ be the map constructed by the product of the composition of $\iota_1,\iota_2$ and $ \iota_3$ between $BSU(n)$ and $BSp(n)$, and $\iota$ between the products of $\CP^{\infty}$.
We define a functor $F\colon P \rightarrow \Top$ as the followings.
\begin{itemize}
    \item For $\sigma \in P$, $F(\sigma )=\prod_{i}X_{\sigma\cap A_i}$.
    \item For $\sigma, \tau \in P$ with $\sigma \subset \tau$, the map between $F(\sigma) \rightarrow F(\tau)$ is defined by the product 
    \[\prod_{i} f_{\sigma\cap A_i, \tau\cap A_i}\colon \prod_i X_{\sigma\cap A_i}\rightarrow \prod_i X_{\tau\cap A_i}.\]
\end{itemize}
We define $X=\hocolim_{P} F$, then the following lemma holds.
\begin{lemma}\label{construction_pre}
    Under the above notation, the cohomology ring of $X$ is isomorphic to $SR(K,\phi)$.
\end{lemma}
\begin{proof}
    We prove this lemma by induction of $|P|$.   
    Let $\sigma$ be a maximal simplex in $K$.
    Let $K'$ be the simplicial complex consisted by the faces of simplices in $P \setminus \{\sigma\}$.
    Then by the assumption of the induction, $H^*(\hocolim_{P\setminus \{\sigma\}}F;\Z) \cong SR(K',\phi)$, $H^*(\hocolim_{P_{\leq \sigma}}F; \Z)\cong \Z[\sigma]$ and $H^*(\hocolim_{P_{<\sigma}}F; \Z)\cong SR(K'\cap \sigma,\phi)$, where $K'\cap \sigma$ is a simplicial complex consisted by the simplices that a simplex in $K'$ and a face of $\sigma$.
    By Lemma \ref{pushout} $X$ is represented by the following homotopy pushout diagrams
    \[
        \xymatrix{
            \hocolim_{P_{<\sigma}} F \ar[r] \ar[d] & \hocolim_{P\setminus \{\sigma\}} F \ar[d]\\
            \hocolim_{P_{\leq \sigma}}F \ar[r]& X.
        }
    \]
    Since $\iota_1,\iota_2, \iota_3$ are the maps that send each generator to its corresponding generator or $0$ in cohomology, the maps in the upper diagram satisfy the condition in Lemma \ref{kernel}.
    Therefore by Lemma \ref{kernel}, we obtain that $H^*(X)\cong SR(K,\phi)$. 
\end{proof}

\begin{proof}[Proof of Theorem \ref{hocolim1}]
    By this discussion, we apply Lemma \ref{construction_pre} to the case $P=P_{\max(K)}$, completing the proof.
\end{proof}

When the degree of generators in $SR(K,\phi)$ are only $2$ and $4$, Theorem \ref{hocolim1} becomes a well-known result.
This corollary is directly proved by the result of Davis and Januszkiewicz \cite{DJ}, and a special case of Theorem 2.34 in \cite{BBCG}.
\begin{corollary}
    Let  $SR(K,\phi)$ be the graded Stanley-Reisner ring for a simplicial complex $K$ with the vertex set $V$ and $\phi \colon V\rightarrow 2\Z_{>0}$.
    When the image of $\phi$ is in $\{2,4\}$, we can construct a space $X$ such that $H^*(X)\cong SR(K,\phi)$.
\end{corollary}

When in $SR(K,\phi)$ there is no pair of generator $x,y\in V$ such that $|x|=|y|=4$ and $xy\ne 0$, we don't have to take the decomposition of the vertex set.
In this case, we can restate Theorem \ref{hocolim1} as follows.
\begin{corollary}\label{hocolim2}
    Let $SR(K,\phi)$ be the graded Stanley-Reisner ring for a simplicial complex $K$ with the vertex set $V$ and $\phi \colon V\rightarrow 2\Z_{>0}$.
    We assume that there is no pair of generator $x,y\in V$ such that $|x|=|y|=4$ and $xy\ne 0$ in $SR(K,\phi)$.
    Then if $SR(K,\phi)$ satisfies the following condition, we can construct a space $X$ such that $H^*(X)\cong SR(K,\phi)$. 
    \begin{itemize}
        \item For $\sigma\in P_{\max}(K)$, the set $\{\phi(x)\mid x\in \sigma , \phi(x)\}$ is equal to $\{2,2, \dots , 2\}$, $\{4,6,\dots ,2n+2\}\cup \{2,2,\dots , 2\}$ or $\{4,8,\dots ,4n\}\cup \{2,2,\dots , 2\}$ as a multiset.
    \end{itemize}
\end{corollary}


\section{Examples}\label{s_example}

In this section we look at some examples about Corollary \ref{hocolim2}. 

Let $SR[K,\phi]\cong \Z[x_4,x_{6},x_{8}]/(x_{6}x_{8})$.
Then the corresponding diagram is the following
\begin{center}
    \[
        BSU(3) \leftarrow BSp(1) \rightarrow  BSp(2).
    \]
\end{center}

Let $SR[K,\phi]\cong \Z[x_4,x_{6,1},x_{6,2},\dots ,x_{6,n},x_{8}]/(x_{6,j}x_{6,k} \text{ for }j\ne k)$, where $|x_{i,j}|=i$.
Then the corresponding diagram is the following
\begin{center}
        \[
            \xymatrix{
                & BSp(2) \ar[dl] \ar[d] \ar[dr] \ar[drr]& & \\
                BSU(4)& BSU(4)& \dots & BSU(4).
            }
        \]
\end{center}

Let $SR[K,\phi]\cong \Z[x_4,x_{6,1},x_{6,2}, x_{8,1},x_{8,2}]/(x_{6,1}x_{6,2},x_{8,1}x_{8,2})$, where $|x_{i,j}|=i$.
Then the corresponding diagram is the following
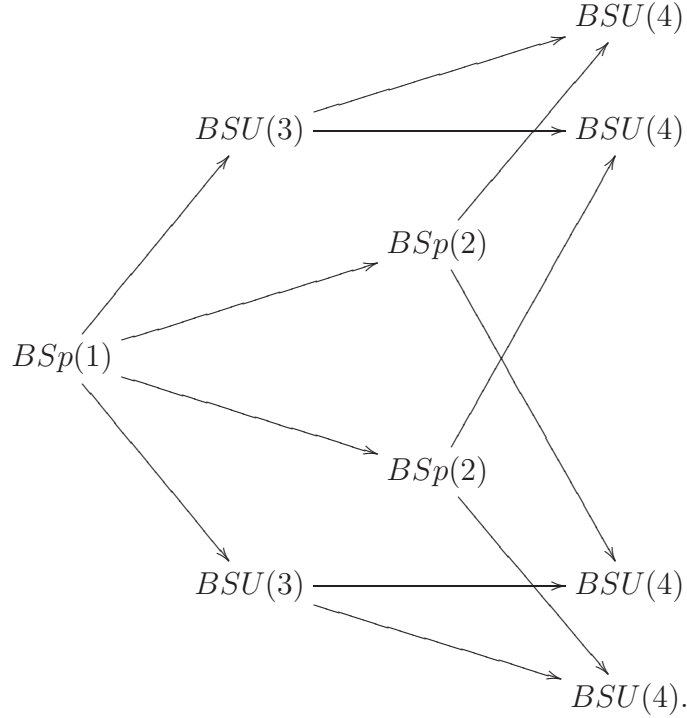
\[
    \xymatrix{
        &  & & BSU(4) \\
        & BSU(3) \ar[rr] \ar[urr]& & BSU(4)\\
        & & BSp(2)\ar[uur]  \ar[dddr]  & \\
        BSp(1) \ar[urr]\ar[drr]\ar[uur] \ar[ddr] &  &   &\\
        & & BSp(2)\ar[uuur]  \ar[ddr]  & \\
        & BSU(3) \ar[rr] \ar[drr]& & BSU(4)\\
        & & & BSU(4).
    }
\]

\section{Approach from algebra over the Steenrod algebra}\label{s_Steenrod}

This section discusses when a graded polynomial ring has an unstable algebra structure over mod-$p$ Steenrod algebra by using previous results.
All of the properties in this section are similar to the properties used by Aguad\'e in \cite{A}.
In the paper, Aguad\'e obtains that which polynomial algebra over $\Z$ is realizable as the integral cohomology ring of a space when the orders of the generators are all different.
To prove this, Aguad\'e observes which polynomial ring has an unstable algebra structure over the mod-$p$ Steenrod algebra by using the result of Adams and Wilkerson\cite{AW}.
In this section, we consider that which polynomial ring has an unstable algebra structure over the mod-$p$ Steenrod algebra under the condition that there is at most $1$ generator with degree $4$.

When a commutative graded algebra $A^*$ over $\Z/p$ has an action of mod-$p$ Steenrod algebra with Cartan formula, we say $A^*$ an algebra over the mod-$p$ Steenrod algebra.
An algebra over the mod-$p$ Steenrod algebra $A^*$ with $A^{2i+1}=0$ for all $i$ is unstable if and only if for all homogeneous elements $x\in A^{2d}$, there are equations
\[
      \Pe^{k}(x)=\begin{cases}
          x^p \quad &k=d\\
          0 \quad &k>d
      \end{cases} \quad (p\geq 3), \quad \text{or} \quad 
      \Sq^{2k}(x)=\begin{cases}
        x^2 \quad &k=d\\
        0 \quad &k>d.
    \end{cases}  \quad (p=2).
\]
When the odd degree parts of $A^*$ is not equal to $0$, the definition of the unstable condition is more complicated.

The following theorem can be obtained by combining the Theorem 1.1 and Theorem 1.2 in Adams and Wilkerson\cite{AW}.
\begin{theorem}[cf. Theorem 1.1 and Theorem 1.2 in Adams and Wilkerson\cite{AW}]\label{invariant}
        Let $A^*$ be a graded polynomial algebra over $\Z/p$ for prime $p$.
        We assume that the following conditions.
        \begin{itemize}
            \item $A^*$ is an unstable algebra over the mod-$p$ Steenrod algebra.
            \item $A^*$ is evenly generated.
            \item $A^*$ is finitely generated as ring.
            \item The degree of generators in $A^*$ are prime to $p$.
        \end{itemize} 
        Then there is an isomorphism 
        \[A^*\cong H^*(BT^n;\Z/p)^W\]
        for some $n$ and a group $W$ generated by pseudoreflections.
\end{theorem}

By using this theorem, we can prove the next theorem.

\begin{proposition}[cf. Proposition 2 in Aguad\'e\cite{A}]\label{Aguade}
    Let $A^*$ be a graded polynomial algebra over $\Z$ satisfying the following condition.
    \begin{itemize}
        \item There is a number $N$ such that for all prime number $p>N$ $A\otimes \Z/p$ has unstable algebra structure over the mod-$p$ Steenrod algebra.
    \end{itemize}
    Then the degree of the generator in $A^*$ is the union of the following table.
    \begin{itemize}
        \item $\{2\}$
        \item $\{4,6,\dots ,2n\}$
        \item $\{4,8,\dots ,4n\}$
        \item $\{4,8,\dots ,4(n-1),2n\}$ for $n\geq 4$
        \item $\{4,12\}$
        \item $\{4,12,16,24\}$
        \item $\{4,10,12,16,18,24\}$
        \item $\{4,12,16,20,24,28,36\}$
        \item $\{4,16, 24,28,36,40,48,60\}$
        \item $\{4,16\}$
        \item $\{4,24\}$
        \item $\{4,48\}$
    \end{itemize}
\end{proposition}
We prove this proposition by the same method in the proof of Proposition 2 in \cite{A}.
\begin{proof}
    Let $p_1,\dots ,p_i$ be the primes larger than $7$ which divide the degree of generators in $A^*$.
    Then by a theorem of Dirichlet we can take a prime number $p>N$ such that 
    \begin{align*}
        &p\equiv 7 \mod 16 \\
        &p\equiv 2 \mod 3 \\
        &p\equiv 3 \mod 5 \\
        &p\equiv 3 \mod 7 \\
        &p\equiv 2 \mod p_i.
    \end{align*}
    By Theorem \ref{invariant}, $A^*\otimes \Z/p$ is isomorphic to an invariant ring $H^*(BT^n;\Z/p)^W$ for some $n$ and a group $W$ generated by pseudoreflections.
    By the classification theorem of $p$-adic pseudoreflection group (cf. Clark and Ewing \cite{CE}), we obtain this proposition.
\end{proof}

For a graded algebra $A^*$, we denote $QA^*=A^*/(A^*_+)^2$.
The following lemma is proved by Thomas.
\begin{theorem}[Theorem 1.4 in Thomas\cite{T}]\label{relation}
    Let $A^*$ be a finitely generated polynomial algebra over $\Z/2$ and an unstable algebra over the mod $2$-Steenrod algebra.
    Then for any number $i$ and odd number $n\geq 3$, the map $$\Sq^{2^i} \colon QA^{2^i(n-1)}\rightarrow QA^{2^in}$$ is a surjection.
\end{theorem}

\begin{lemma}\label{remove1}
    Let $A^*$ be a polynomial algebra over $\Z$ such that the degrees of generators are equal to one of the following list  as a multiset.
    \begin{itemize}
        \item $\{4,8,\dots ,4(n-1),2n\}\cup\{2,2,\dots 2\}$ with $n\geq 4$ and $n$ is not a power of $2$.
        \item $\{4,12\}\cup\{2,2,\dots 2\}$
        \item $\{4,12,16,24\}\cup\{2,2, \dots 2\}$
        \item $\{4,10,12,16,18,24\}\cup\{2,2,\dots 2\}$
        \item $\{4,12,16,20,24,28,36\}\cup\{2,2,\dots 2\}$
        \item $\{4,16, 24,28,36,40,48,60\}\cup\{2,2,\dots 2\}$
        \item $\{4,24\}\cup\{2,2, \dots 2\}$
        \item $\{4,48\}\cup\{2,2, \dots 2\}$
    \end{itemize}
    Then $A^*\otimes \Z/2$ doesn't have an unstable algebra structure over the mod-$2$ Steenrod algebra.
\end{lemma}
\begin{proof}    
    We assume that $A^*\otimes \Z/2$ has an unstable algebra over the mod-$2$ Steenrod algebra.
    By Theorem \ref{relation}, if there is a generator $x$ with $|x|=12$, then there must be a generator $y$ with $|y|=8$.
    Therefore the second, third, fourth and fifth cases don't have an unstable algebra structure over the mod-$2$ Steenrod algebra.

    Similarly, if there is a generator $x$ with $|x|=60,24,48$, then there must be a generator $y$ with $|y|=56,16,32$, respectively.
    Therefore $sixth,seventh$ and $eighth$ cases don't have an unstable algebra structure over the mod-$2$ Steenrod algebra.

    It remains to show the first case. 
    In this case we can denote $n=2^im$ for an integer $i$ and an odd number $m\geq  3$.
    When $i=0$, by Proposition 3 in \cite{A}, $A^*$ doesn't have an unstable algebra structure over the mod-$2$ Steenrod algebra.
    When $i\geq 1$, by Theorem \ref{relation} $\Sq^{2^{i+1}}\colon QA^{2^{i+1}(m-1)} \rightarrow QA^{2^{i+1}m}$ must be a surjection.
    But $\mathrm{dim}(QA^{2^{i+1}(m-1)})=1$ and $\mathrm{dim}(QA^{2^{i+1}m})=2$, it is contradiction.
    Therefore the first case doesn't have an unstable algebra structure over the mod-$2$ Steenrod algebra.

    Combining these discussion, complete the proof.
\end{proof}

\begin{lemma}\label{remove2}
    Let $A^*$ be a polynomial algebra over $\Z$ such that the degrees of generators are equal to $\{4,16\}\cup\{2,2,\dots 2\}$ as a multiset.
    Then $A^*\otimes \Z/3$ doesn't have an unstable algebra over the mod-$3$ Steenrod algebra.
\end{lemma}
\begin{proof}
    Let $A^*$ be the polynomial ring with the degrees of generators are equal to $\{4,16\}\cup\{2,\dots 2\}$, and let $x$ be the generator with degree $16$ in $A^*$.
    We assume that $A^*\times \Z/3$ has an unstable algebra structure over the mod-$3$ Steenrod algebra.
    By Adem relation, there is an equation $\Pe^8=-\Pe^1\Pe^7$.
    Since $\Pe^8(x)=x^3$, $x^3$ is in $\mathrm{Im}(\Pe^1)$.
    On the other hand since there is no generator $y$ with $|y|\equiv 12 \mod 16$, the term $x^i$ is not included in the image $\Pe^1$.
    It is contradiction.
\end{proof}

\begin{proposition}\label{polynomial_with_Steenrod}
    Let $A^*$ be a non-trivial graded polynomial algebra over $\Z$ satisfying the following condition.
    \begin{itemize}
        \item There is at most one generator with degree $4$.
        \item For all prime numbers $p$, $A\otimes \Z/p$ has an unstable algebra structure over the mod-$p$ Steenrod algebra.
    \end{itemize}
    Then the degree of the generator in $A^*$ is equal to the one of the following table as a multiset  for some $n$.
    \begin{itemize}
        \item $\{2,2,\dots 2\}$
        \item $\{4,6,\dots ,2n\}\cup\{2,2,\dots 2\}$
        \item $\{4,8,\dots ,4n\}\cup\{2,2, \dots 2\}$
        \item $\{4,8,\dots ,2^{n+1}-4,2^n\}\cup\{2,2,\dots 2\}$
    \end{itemize}
\end{proposition}

\begin{proof}
    By Proposition \ref{Aguade} and the first condition, the degree of the generator in $A^*$ is equal to the union of the one of the table in Proposition \ref{Aguade} and the copies of $\{2\}$.
    By Lemma \ref{remove1}, \ref{remove2}, the cases except for the cases $\{4,6,\dots ,2n\}$, $\{4,8,\dots ,4n\}$ or $\{4,8,\dots ,2^{n+1}-4,2^n\}$ don't satisfies the second condition.
    Thus we obtain this proposition.
\end{proof} 

\begin{example}
    Let $\Z/2[t_1, \dots t_{2^n}]\cong H^*(BT^{2^n};\Z/2)$ for $n\geq 2$.
    Then we take a subring of $H^*(BT^{2^n};\Z/2)$ as
    \[
        \Z/2[t_1,t_2, \dots t_{2^{n}-1}]^{W(Sp(2^{n}-1))} \otimes \Z/2[{t_{2^n}}^{2^{n-1}}],
    \]
    where $W(Sp(2^{n}-1))$ is the Weyl group of $Sp(2^n-1)$ and $\Z/2[t_1,t_2, \dots t_{2^n}]^{W(Sp(2^{n}-1))}$ is the invariant ring of the canonical $W(Sp(2^{n}-1))$-action.
    Since $\Z/2[t_1,t_2, \dots t_{2^n}]^{W(Sp(2^{n}-1))}$ is isomorphic to $H^*(BSp(2^n-1);\Z/2)$, this subring preserve the action of mod-$2$ Steenrod operations, and the degree of generators in this subring is $\{4,8,\dots ,2^{n+1}-4,2^n\}$.
    This subring has the unstable algebra structure over the mod-$2$ Steenrod algebra induced by $H^*(BT^{2^n};\Z/2)$.
\end{example}
When $p$ is an odd prime number, the cohomology ring $H^*(BSpin(2^n);\Z/p)$ is isomorphic to the polynomial ring with generator's degree $\{4,8,\dots ,2^{n+1}-4,2^n\}$ (cf.  Chapter III, Theorem 3.19 in \cite{TM}), and has the unstable algebra structure over the mod-$p$ Steenrod algebra.
And the ring in this example has the unstable algebra structure over the mod-$2$ Steenrod algebra.
Therefore by only using the method in this section, we cannot remove the case $\{4,8,\dots ,2^{n+1}-4,2^n\}$ in Proposition \ref{polynomial_with_Steenrod}.


\section{Stanley-Reisner ring and Steenrod algebra}\label{s_requirement}

Let $K$ be a simplicial set with the vertex set $V$, and $\phi \colon V\rightarrow 2\Z_{>0}$.
For a polynomial $f \in SR(K,\phi)$ and a monomial $g$, we denote $g < f$ when $g\ne 0$ in $SR(K,\phi)$ and coefficient of $g$ in $f$ is not equal to 0. 
This notation is well-defined because the ideal $I$ is generated by monomials.

\begin{lemma}\label{preserve_pre}
    Let $X$ be a space such that $H^*(X)\cong SR(K,\phi)$ for some graded Stanley-Reisner ring, and $\sigma \in K$ be a maximal simplex.
    Then for any prime number $p$ the ideal $(V\setminus \sigma)$ in $H^*(X;\Z/p)\cong SR(K,\phi)\otimes \Z/p$ preserves the action of the mod-$p$ Steenrod algebra.
\end{lemma}
\begin{proof}
    
    We assume that the ideal $(V\setminus \sigma)$ doesn't preserve the action of mod-$p$ Steenrod algebra.
    Then by the Cartan formula, there is $x\in V\setminus \sigma$ and a monomial $f \in SR(K,\phi)$ such that $f \notin (V\setminus \sigma)$ and $f < \Pe^i(x)$ for some $i$.
    Now we can take $f$ with $i$ being minimal i.e. $\Pe^j(x)\in (V\setminus \sigma)$ for $j<i$.
    We denote $g=\prod_{y\in \sigma}y$.
    Since $f$ is a monomial generated by $\sigma$, $fg\ne0$ in $SR(K,\phi)$.
    Then $$f g < \Pe^i(x)g$$ and since $i$ is minimal, $$fg \not< \sum_{j>0}\Pe^{i-j}(x)\Pe^{j}\left(g\right).$$
    Therefore by Cartan formula $$f g < \Pe^i(x)g+\sum_{j>0}\Pe^{i-j}(x)\Pe^{j}\left(g\right)=\Pe^i\left(xg\right),$$ and we obtain $$\Pe^i\left(xg\right)\ne 0.$$
    Since $xg=0$ in $SR(K,\phi)$, it is contradiction, so the assumption is false. 
    This completes the proof.
\end{proof}

\begin{proposition}\label{preserve}
    Let $X$ be a space such that $H^*(X)\cong SR(K,\phi)$ for some graded Stanley-Reisner ring.
    Let $\sigma_1, \dots \sigma_m \in K$ be maximal simplexes. 
    Then for any prime number $p$ the ring
    \[
        SR(K,\phi)\otimes \Z/p\Z /(V\setminus \sigma_1\cap \dots \cap \sigma_m)
    \]
    has an unstable algebra structure over the mod-$p$ Steenrod algebra induced by the quotient map 
    \[H^*(X;\Z/p)\cong SR(K,\phi)\otimes \Z/p\Z \rightarrow SR(K,\phi)\otimes \Z/p\Z /(V\setminus \sigma_1\cap \dots \cap \sigma_m).\]
\end{proposition}
\begin{proof}
    By Lemma \ref{preserve_pre} for all $x\in V\setminus \sigma_k$ and $i$, we obtain $$\Pe^i(x)\in (V\setminus \sigma_k)\subset (V\setminus \sigma_1\cap\dots \cap \sigma_n).$$
    Therefore the ideal $(V\setminus \sigma_1\cap \dots \cap \sigma_m)$ preserves the action of the mod-$p$ Steenrod algebra.
\end{proof}

\begin{theorem}\label{requirement}
    For a graded Stanley-Reisner ring $SR(K,\phi)$, let $X$ be a space such that $H^*(X)\cong SR(K,\phi)$.
    We assume that there is no pair of generators $x,y\in V$ such that $\phi(x)=\phi(y)=4$ and $xy\ne 0$ in $SR(K,\phi)$.
    Then for $\sigma \in P_{\max}(K)$, the set $\{\phi(x)\mid x \in \sigma\}$ is equal to $\{2,\dots 2\}$, $\{4,6,\dots 2n+2\}\cup\{2,\dots 2\}$, $\{4,8, \dots 4n\}\cup\{2,\dots 2\}$ or $\{4,8, \dots 2^{n+2}-8, 2^{n+2}-4, 2^{n+1}\}\cup\{2,\dots 2\}$ as a multiset for some $n\geq 1$.
\end{theorem}
\begin{proof}
    Since $\sigma \in K$, there is no relation between the generators in $\sigma$.
    Therefore there is an isomorphism
    \[
        SR(K,\phi)/(V\setminus \sigma)\cong \Z[\sigma].
    \]
    By the definition of $P_{\max}(K)$ there are maximal simplexes $\sigma_1,\dots ,\sigma_m \in K$ such that $\sigma=\bigcap_i \sigma_i$.
    Thus $\Z[\sigma]$ satisfies the condition of Proposition \ref{preserve}.
    By the assumption in the statement, for any $\sigma \in P_{\max}(K)$ there is at most one generator with degree $4$ in $\sigma$.
    By Proposition \ref{preserve} and this condition, the polynomial ring $\Z[\sigma]$ satisfies the condition in Proposition \ref{polynomial_with_Steenrod}.
    Therefore the set $\{\phi(x)\mid x \in \bigcap_i \sigma_i\}$ is equal to $\{2,\dots 2\}$, $\{4,6,\dots 2n+2\}\cup\{2,\dots 2\}$, $\{4,8, \dots 4n\}\cup\{2,\dots 2\}$ or $\{4,8, \dots 2^{n+2}-8, 2^{n+2}-4, 2^{n+1}\}\cup\{2,\dots 2\}$ as a multiset.
\end{proof}

\begin{example}
    Let $SR(K,\phi)\cong \Z[x_4,x_6]/(x_4x_6)$, where $|x_i|=i$. Then $P_{\max}(K)=\{\{4\},\{6\}\}$. By Theorem \ref{requirement}, there is no space $X$ with $H^*(X)\cong \Z[x_4,x_6]/(x_4x_6)$.
\end{example}


\section{Proof of the main theorem}\label{s_proof}
By combining Corollary \ref{hocolim2} and Theorem \ref{requirement}, we can prove Theorem \ref{main}.

\begin{proof}[Proof of Theorem \ref{main}]
    In Corollary \ref{hocolim2}, we prove that if $SR(K,\phi)$ satisfies these conditions then there is a space $X$ such that $H^*(X)\cong SR(K,\phi)$.
    
    On the other hand, we assume that there is a space $X$ such that $H^*(X)\cong SR(K,\phi)$.
    By assumption in the statement for $i=2$, $SR(K,\phi)$ satisfies the condition of Theorem \ref{requirement}.
    By Theorem \ref{requirement}, for any $\sigma \in P_{\max}(K)$ the set $\{\phi(x)\mid x \in \sigma\}$ is equal to $\{2,\dots 2\}$, $\{4,6,\dots 2n+2\}\cup\{2,\dots 2\}$, $\{4,8, \dots 4n\}\cup\{2,\dots 2\}$ or $\{4,8, \dots 2^{n+2}-8, 2^{n+2}-4, 2^{n+1}\}\cup\{2,\dots 2\}$ as a multiset.
    By the assumption in the statement, there is no pair of generators $x,y$ such that $|x|=|y|=2^n$ for some $n\geq 3$ and $xy\ne 0$ in $SR(K,\phi)$. 
    Since the case $\{4,8, \dots 2^{n+2}-8, 2^{n+1}-4, 2^{n}\}\cup\{2,\dots 2\}$ for $n\geq 3$ includes such a pair of generators, this case doesn't appear.
    Therefore for any $\sigma \in P_{\max}(K)$ the set $\{\phi(x)\mid x \in \sigma\}$ is equal to $\{2,\dots 2\}$, $\{4,6,\dots 2n+2\}\cup\{2,\dots 2\}$, $\{4,8, \dots 4n\}\cup\{2,\dots 2\}$ as a multiset.

    By combining these, the proof is complete.
\end{proof}

\end{document}